\documentclass{amsproc}
\copyrightinfo{2009}{American Mathematical Society}

\usepackage{enumerate}
\usepackage{stackrel}
\usepackage{mathrsfs}
\usepackage{xcolor}
\usepackage{soul}
\usepackage{hyperref}
\usepackage{bm}
\usepackage{comment}

\hypersetup{
  colorlinks = true,
  urlcolor = blue,
  linkcolor = black,
  citecolor = red
}

\newtheorem{theorem}{Theorem}[section]
\newtheorem{lemma}[theorem]{Lemma}

\theoremstyle{definition}

\theoremstyle{remark}
\newtheorem{remark}[theorem]{Remark}
\numberwithin{equation}{section}

\newcommand{\hyp}[5]{{}_{#1}F_{#2}\!
\left(\genfrac{}{}{0pt}{}{#3}{#4};#5\right)}
\newcommand{\pe}[2]{\left\langle#1,#2\right\rangle}

\newcommand{\N}{{\mathbb N}}

\newcommand{\expe}{{\mathrm e}}
\newcommand{\dd}{{\mathrm d}}

\definecolor{Mycolor2}{HTML}{e85d04}
\sethlcolor{black}

\begin{document}
\title{Orthogonality of the big $-1$ Jacobi polynomials for non-standard parameters}
\date{\today}
\author[H.~S.~Cohl]{Howard S.~Cohl}
\address{Applied and Computational 
Mathematics Division, National Institute of Standards 
and Tech\-no\-lo\-gy, Gaithersburg, MD 20899-8910, USA}
\email{howard.cohl@nist.gov}
\author[R.~S.~Costas-Santos]{Roberto S.~Costas-Santos}
\address{Department of Quantitative Methods, Universidad Loyola Andaluc\'ia, Sevilla, Spain}
\email{rscosa@gmail.com}
\date{\today}
\begin{abstract}
The big $-1$ Jacobi polynomials $(Q_n^{(0)}(x;\alpha,\beta,c))_n$ have been classically defined for 
$\alpha,\beta\in(-1,\infty)$, $c\in(-1,1)$.
We extend 
this family so that wider sets of parameters are allowed, i.e.,
they are non-standard.
Assuming initial conditions $Q^{(0)}_0(x)=1$, $Q^{(0)}_{-1}(x)=0$,
we consider the big $-1$ Jacobi polynomials as 
monic orthogonal polynomials 
which therefore satisfy the following three-term 
recurrence relation
\[
xQ^{(0)}_n(x)=Q^{(0)}_{n+1}(x)+b_{n} Q^{(0)}_n(x)+ u_{n} Q^{(0)}_{n-1}(x), 
\quad n=0, 1, 2,\ldots.
\]
For standard parameters, the coefficients $u_n>0$ 
for all $n$.
We discuss the situation where Favard's theorem cannot be directly applied for some positive integer 
$n$ such that $u_n=0$.
We express the big $-1$ Jacobi polynomials for 
non-standard parameters as a product of two polynomials. 
Using this factorization, we obtain a bilinear form with 
respect to which these polynomials are orthogonal.
\end{abstract}

\maketitle
\section{Introduction}
In 2011, Vinet and Zhedanov in \cite{MR2770369} 
obtained a new family of ``classical'' orthogonal
polynomials which can be obtained from the little 
$q$-Jacobi polynomials in the limit $q\to-1$. 
By ``classical'' it is meant
that these polynomials 
satisfy  (apart from a three-term recurrence relation) 
a nontrivial eigenvalue equation of the form
\[
{\mathcal L}(p_n(x))=\lambda_n p_n(x),
\]
where $\mathcal L$ is a linear differential-difference 
operator  which is of first order in the derivative  $\partial_x$ and contains 
a reflection operator $R$ which acts as $Rf(x)=f(-x)$.
They referred to these polynomials as being in the continuous $-1$ hypergeometric 
orthogonal polynomial scheme. This was illustrated by using a limiting process from the 
Askey--Wilson polynomials to the Bannai--Ito polynomials. 
In \cite{MR2931336} the same authors applied this 
limit by starting with the big 
$q$-Jacobi polynomials in order to obtain hypergeometric representations for the big $-1$ Jacobi polynomials.

The polynomials the authors have considered in this $-1$ hypergeometric orthogonal polynomial scheme
are $q$-polynomials for standard parameters, i.e., 
those values for the parameters whose 
weight function 
related to these families is positive definite 
(see \cite{pelletieretall2023}). 
The main aim of this work is to investigate the property 
of orthogonality for the big $-1$ polynomials 
for those values of the parameters, $\alpha$, and $\beta$, for which the coefficient $u_n$ in the recurrence relation \eqref{eq:TTRRgen} relative to these polynomials can be equal to zero for some 
$n\in \mathbb N_0$.

In the last three decades, some of the classical 
orthogonal polynomials with non-classical parameters 
have been provided with certain non-standard 
orthogonality pro\-per\-ties.
In the pioneering work by Kwon and Littlejohn in 1995 \cite{mr1343537}, Sobolev orthogonality for the Laguerre 
polynomials $L_n^{(-N)}(x)$ with $N\in\mathbb{N}$, 
was obtained using the following inner product:
\begin{equation*}
\pe f g=\sum_{m=0}^{N-1}\sum_{j=0}^m B_{m,j}(N)
\left(f^{(m)}(0)g^{(j)}(0)+f^{(j)}(0)g^{(m)}(0)
\right)
+\int_0^{\infty} f^{(N)}(x)g^{(N)}(x)\,\expe^{-x}\,\dd x,
\end{equation*}
where the superscripts represent differentiation and 
\[
B_{m,j}(N):=\left\{\begin{array}{ll}
\displaystyle \sum_{p=0}^j(-1)^{m+j}\binom{N-1-p}{m-p}\binom{N-1-p}{j-p},& 0\le j\le m\le N-1, \\[7mm]
\displaystyle \frac 12\sum_{p=0}^m\binom{N-1-p}{m-p}^2,& 0\le j=m\le N-1.
\end{array}
\right.
\]
In 1998, \'Alvarez de Morales and co-authors 
\cite{mopepi2} found, using a different technique, 
the orthogonality 
for ultraspherical polynomials $C_n^{-N+\frac12}$, 
where $N\in\N$. 
Later in \cite{mr1917278}, M.~Alfaro et al.~considered 
the cases for the Jacobi polynomials in which both 
parameters, $\alpha$, and $\beta$, were negative integers, 
proving that in such a case the Jacobi polynomials 
satisfy a Sobolev orthogonality.
In \cite{mr1795525}, M.~Alfaro et al.~considered 
the situation where the inner product can be written 
in the form
\begin{equation}\label{OrtgBilin}
\langle f, g\rangle = F^t A G+\int f^{(N)}(x)g^{(N)}(x)\,\mathrm d\mu(x)\,,
\end{equation}
where $F$ and $G$ are vectors obtained by evaluating 
$f$ and $g$ and maybe their derivatives at some points, 
$A$ is a symmetric real matrix, and $\mathrm d\mu$ is the 
orthogonality measure associated with the $N$th 
derivative of either the Laguerre, ultraspherical 
or Jacobi polynomials.

Note that for the parameters considered in this situation 
there exists a $n=N\in\N_0$ so that the coefficient 
$u_n$ in the recurrence relation 
\begin{equation}\label{eq:TTRRgen}
xp_n=p_{n+1}+b_n p_n+u_n p_{n-1}\,,\qquad n=0, 1, 2, ....
\end{equation}
vanishes, i.e., $u_N=0$. 
We say a non-standard parameter (or set of parameters) is degenerate if there exists at least one positive integer $N$ for which $u_N=0$.

The first term in the inner product \eqref{OrtgBilin} plays the role 
of the orthogonality 
for the set of polynomials of degree less than $N$, 
since if $f, g$ are two polynomials, 
with $\deg f$, $\deg g<N$, then $f^{(N)}(x)=g^{(N)}(x)=0$, and one has
\[
\langle f, g\rangle=F^t A G.
\]
In this term, the points for the evaluation are the 
roots of $p_N$; this ensures that the inner product vanishes 
when one entry is $p_n$ with $n\geq N$, since $p_N$ 
is a factor of any $p_n$ with $n\geq N$.
The second term is relevant for polynomials with 
degree greater than $N$, and it is  needed in order 
to have an orthogonality
characterizing  the sequence $(p_n)_{n=0}^{\infty}$.
The technique used in \cite{mopepi2} is applicable 
for all classical orthogonal polynomials where
$u_n$ vanishes at certain $n=N\in \mathbb N_0$,
and in fact it has been used, among others, in 
\cite{mr2560989,mr2643877,mr2581384,mr2848982}.

In 2009 \cite{cola1}, 
Costas-Santos and 
Sanchez-Lara studied 
this problem for classical discrete polynomials, i.e., discrete 
polynomials in the Askey Scheme of generalized hypergeometric orthogonal polynomials. We consider classical discrete polynomials 
for which there exists $N$ for which $u_N=0$, 
i.e., for non-standard degenerate parameters.
Note that 
this can only happen for the Racah, Hahn, dual Hahn, and 
Krawtchouk polynomials which are the Wilson, continuous 
Hahn, continuous dual Hahn and Meixner polynomials for 
which some of its parameters are equal to a negative integer.
The corresponding three-term recurrence relation for 
these polynomials also presents one vanishing coefficient 
i.e., $u_N=0$, and the inner product
found can be written as
\begin{equation}\label{OrtgLin}
\langle f, g\rangle = \int f(x)g(x)\,\mathrm d\mu_d(x)
+\int f^{(N)}(x)g^{(N)}(x)
\,\mathrm d\mu_c(x)\,,
\end{equation}
where $\mathrm d\mu_d$ is a discrete measure with a finite number 
of masses and $\mathrm d\mu_c$ is an absolutely continuous measure. 
For more details and further reading see 
\cite{cola1,cola2,mr2149265,mr1405981,mr3619764} and references 
therein.
The organization of the paper is as follows. In Section \ref{sec2} 
we provide some preliminary material
and in Section 
\ref{sec3} we obtain the property of orthogonality for the big $-1$ Jacobi polynomials for non-standard degenerate parameters.
\section{Preliminary material and notations}\label{sec2}
We adopt the following notations:~$\mathbb N_0:=\{0\}\cup\mathbb N=\{0, 1, 2, ...\}$, and 
we use the sets $\mathbb P$, $\mathbb P'$, 
to represent the linear space of polynomials 
with complex coefficients and its algebraic dual space.
We denote by $\langle {\bf u},p\rangle$ the duality 
bracket for ${\bf u}\in \mathbb P'$ and $p\in \mathbb P$.
The big $-1$ Jacobi polynomials have representations given in terms of the Gauss 
hypergeometric function 
defined as \cite[\href{http://dlmf.nist.gov/15.2.E1}{(15.2.1)}]{NIST:DLMF} 
\begin{equation}
\hyp21{a,b}{c}{z}:=\sum_{n=0}^\infty \frac{(a)_n(b)_n}
{(c)_n}\frac{z^n}{n!},
\end{equation}
where $|z|<1$ and the shifted factorial is defined 
as \cite[\href{http://dlmf.nist.gov/5.2.E4}{(5.2.4-5)}]{NIST:DLMF}
$(a)_n:=(a)(a+1)\cdots(a+n-1)$,
$a\in\mathbb C$,
$n\in\N_0$. The following ratio of two gamma functions \cite[\href{http://dlmf.nist.gov/5}{Chapter 5}]{NIST:DLMF} are related to 
the shifted factorial, namely
for $a\in\mathbb C\setminus-{\N}_0$, one
 has
\begin{equation}
(a)_n=\frac{\Gamma(a+n)}{\Gamma(a)},
\label{Pochdef}
\end{equation}
which allows one to extend the definition to non-positive 
integer values of $n$. For a description of the properties of the gamma function, see \cite[\
href{http://dlmf.nist.gov/5}{Chapter 5}]{NIST:DLMF}.


\section{Big $-1$ Jacobi polynomials}\label{sec3}

The monic 
big $-1$ Jacobi polynomials 
can be defined in terms of 
the Gauss hypergeometric function 
as follows:
\begin{equation}
\label{eq:bQdef}
\hspace{0.3cm}Q_n^{(0)}(x):=\!\kappa_n
\left\{ \begin{array}{ll}
\displaystyle \!\!\hyp{2}{1}
{-\frac n2,\frac{n+\alpha+\beta+2}2}{\frac{\alpha+1}2}
{\dfrac{1\!-\!x^2}{1\!-\!c^2}}\!+\!\dfrac{n(1-x)}{(c+1)(\alpha+1)}
\hyp{2}{1}{1-\frac n2,\frac{n+\alpha+\beta+2}2}
{\frac{\alpha+3}2}{\dfrac{1\!-\!x^2}{1\!-\!c^2}},\   
\quad\mathrm{if}\ n\ \mathrm{even},\\[18pt]
\displaystyle \!\hyp{2}{1}{-\frac {n-1}2,\frac{n+\alpha+\beta+1}2}
{\frac{\alpha+1}2}
{\dfrac{1\!-\!x^2}{1\!-\!c^2}}\!-\!\dfrac{(\alpha\!+\!\beta\!+\!n\!+\!1)(1\!-\!x)}
{(1\!+\!c)(\alpha\!+\!1)}
\hyp{2}{1}{-\frac {n-1}2,\frac{n+\alpha+\beta+3}2}
{\frac{\alpha+3}2}{\dfrac{1\!-\!x^2}{1\!-\!c^2}}, \ \mathrm{if}\ n\ \mathrm{odd},
\end{array} \right.
\end{equation}
where $\kappa_n$ is defined as
\begin{equation}
\label{eq:kappan}
\hspace{0.3cm}\kappa_n:=\!
\left\{ \begin{array}{ll}
\displaystyle \!\dfrac{(1-c^2)^{\frac12n}(\frac12(\alpha+1))_{\frac12n}}
{(\frac12(n+\alpha+\beta+2))_{\frac12n}},  
\quad\mathrm{if}\ n\ \mathrm{even},\\[18pt]
\displaystyle \!(c+1)\dfrac{(1-c^2)^{\frac12(n-1)}(\frac12(\alpha+1))_{\frac12(n-1)}}
{(\frac12(n+\alpha+\beta+1))_{\frac12(n-1)}}, \quad\mathrm{if}\ n\ \mathrm{odd}.
\end{array} \right.
\end{equation}


\noindent The big $-1$ Jacobi polynomials 
$Q_n^{(0)}(x;\alpha,\beta,c)$, satisfy the three-term 
recurrence relation \cite[(2.20), (2.21)]{MR2931336}:
\begin{equation}\label{ttrrbmonej}
x Q_n^{(0)}(x)=Q_{n+1}^{(0)}(x)+b_n Q_n^{(0)}(x)+
u_n Q_{n-1}^{(0)}(x),
    \end{equation}
where
\begin{equation}
\label{eq:bn}
\hspace{0.3cm}b_n:=\!
\left\{ \begin{array}{ll}
\displaystyle \!-c+\dfrac{(c-1)n}{\alpha+\beta+2n}+\dfrac{(c+1)(\beta+n+1)}
{\alpha+\beta+2n+2},\quad\mathrm{if}\ n\ \mathrm{even},\\[14pt]
\displaystyle \!c-\dfrac{(c-1)(n+1)}{\alpha+\beta+2n+2}
-\dfrac{(c+1)(\beta+n)}{\alpha+\beta+2n}, \quad\mathrm{if}\ n\ \mathrm{odd},
\end{array} \right.
\end{equation}
and
\begin{equation}
\label{eq:un}
\hspace{0.3cm}u_n:=\!
\left\{ \begin{array}{ll}
\displaystyle \!\dfrac{(c-1)^2n(\alpha+\beta+n)}{(\alpha+\beta+2n)^2},& \quad\mathrm{if}\ n\ \mathrm{even},\\[14pt]
\displaystyle \!\dfrac{(c+1)^2(\alpha+n)(\beta+n)}{(\alpha+\beta+2n)^2}, & \quad\mathrm{if}\ n\ \mathrm{odd},
\end{array} \right.
\end{equation}
with initial conditions $Q^{(0)}_{-1}(x)=0$,
$Q^{(0)}_{0}(x)=1$.
For any real $c\ne 1$ and real $\alpha$, $\beta$ 
satisfying the restriction $\alpha,\beta >-1$, 
the recurrence coefficients $b_n$, $u_n$
are real and positive. Hence the 
big $-1$ Jacobi polynomials are positive definite 
orthogonal polynomials. 

\begin{remark}\label{rem:1}
Observe that if $\alpha$ 
(resp.~$\beta$) is a negative odd 
integer, namely $\alpha=-2N+1$ 
(resp.~$\beta=-2N+1$) then 
$u_{2N-1}=0$, i.e.,~$\alpha$ ( resp. $\beta$) is 
non-standard  degenerate parameter. 
So we can apply the degenerate version of 
Favard’s theorem (see 
\cite[Theorem 2.2]{cola2}), i.e., 
there exist moment functionals, $\mathscr L_0$ and 
$\mathscr L_N$, such that the polynomial sequence 
is orthogonal with respect to the bilinear form
\[
\langle p,r\rangle=\mathscr L_0(pr)+
\mathscr L_N(\mathscr T^{(2N-1)}p \mathscr T^{(2N-1)}r),\quad p, r \in \mathbb P,
\]
where $\mathscr T$ is a certain lowering operator.
\end{remark}

\noindent Taking the former remark into account, 
the Gauss hypergeometric function with a suitable 
normalization factorizes as follows.
\begin{lemma}[Factorization] \label{lem:1}Let $n, N\in \N$, 
$a\in \mathbb C$. Then 
\begin{equation}\begin{split}\label{id:genfact}
(-N+1)_{n+N}\,\hyp{2}{1}{-n-N, a}{-N+1}{x}=&
(N+1)_n (-N+1)_{N}\,\hyp{2}{1}{-N, a}{-N+1}{x}
\hyp{2}{1}{-n,a+N}{N+1}{x}.
\end{split}\end{equation} 
\end{lemma}
\begin{proof}
By definition one has
\[\begin{split}
(-N+1)_{n+N}\,\hyp{2}{1}{-n-N, a}{-N+1}{x}=&
(-N+1)_{n+N}\sum_{k=0}^{n+N} \dfrac{(-n-N,a)_k}
{(-N+1,1)_k x^k}\\
=&\sum_{k=0}^{n+N} \dfrac{(-n-N,a)_k}{(1)_k}
(-N+k+1)_{n+N-k} x^k.
\end{split}
\]
Then setting $k=k+N$ we have
\begin{eqnarray}
\nonumber (-N+1)_{n+N}\,\hyp{2}{1}{-n-N, a}{-N+1}{x}&=&
\sum_{k=0}^{n} \dfrac{(-n-N,a)_{k+N}}{(1)_{k+N}}
(k+1)_{n-k} x^{k+N}\\ \nonumber
&=&\dfrac{(-n-N,a)_N n!}{N!}x^N 
\sum_{k=0}^{n} \dfrac{(-n,a+N)_{k}}{(N+1,1)_{k}}x^{k}\\
&=&(N+1)_n(a)_N(-x)^N\,\hyp{2}{1}{-n,a+N}{N+1}{x}.
\label{id:fact-3}
\end{eqnarray}
Setting $n=0$ in \eqref{id:fact-3},
one obtains
\begin{equation}\label{id:factinit}
(-N+1)_{N}\,\hyp{2}{1}{-N, a}{-N+1}{x}=(a)_N(-x)^N.
\end{equation}
Combining \eqref{id:fact-3} and \eqref{id:factinit}
produces the result.
\end{proof}

Now that we have obtained factorization for the Gauss hypergeometric 
function, we will use it to obtain a factorization for
the big $-1$
Jacobi polynomials when the parameter $\alpha$ 
(resp.~$\beta$) is equal to $-2N-1$ for some positive integer $N$.

\begin{lemma}[Factorization of big $-1$ Jacobi 
polynomials] \label{lem:3}
Let $m, M, N\in \mathbb N_0$, $\alpha, \beta, 
c\in \mathbb C$, $c\ne\pm 1$, with $\alpha=-2N-1$ or 
$\beta=-2N-1$. Then
\begin{eqnarray}
\label{id:bmonejfactal0}
&&Q^{(0)}_{2N+1}(x;-2N-1,\beta,c)=
(x^2-1)^{N}(x-1),\\
\label{id:bmonejfactbe0}
&&Q^{(0)}_{2N+1}(x;\alpha,-2N-1,c)=
(x^2-c^2)^{N}(x+c),\\
\label{id:bmonejfactbe1}
&&Q^{(0)}_{2N+1+m}(x;-2N-1,\beta,c)=(-1)^m
(x^2-1)^{N}(x-1) Q^{(0)}_{m}(-x;2N+1,\beta,-c),\\
\label{id:bmonejfactal1}
&&Q^{(0)}_{2N+1+m}(x;\alpha,-2N-1,c)=
(x^2-c^2)^{N}(x+c)
Q^{(0)}_{m}(x;\alpha,2N+1,-c).
\end{eqnarray} 
\end{lemma}
\begin{proof}
Identities \eqref{id:bmonejfactal0} and 
\eqref{id:bmonejfactbe0} follow from
\eqref{id:bmonejfactal1} and \eqref{id:bmonejfactbe1} 
by setting $m=0$ in these. 
We will first prove 
\eqref{id:bmonejfactbe1} 
and then \eqref{id:bmonejfactal1}.
Consider $m=2M$, $n=2N+1+m$. Then  one has
\begin{eqnarray}
&&\hspace{-1.5cm}Q^{(0)}_{2N+1+m}(x;\alpha,-2N-1,c)
=\kappa_{n}
\biggl(\hyp{2}{1}{-N-M,M+\widehat \beta}{-N}
{\dfrac{1-x^2}{1-c^2}}\nonumber\\
&&\hspace{4cm}+\frac{(M+\widehat\beta)(1-x)}
{(c+1)N}
\hyp{2}{1}{-N-M,M+1+\widehat \beta}{-N+1}
{\frac{1-x^2}{1-c^2}}\biggr),
\end{eqnarray}
where $\widehat \beta=(\beta+1)/2$. Next using Lemma 
\ref{lem:1} and after a tedious calculation comparing the 
final expression with the one for $Q_{m}(-x; 2N+1,\beta, -c)$ 
and its corresponding normalization coefficient, 
the identity for the even $m$ values holds. 
For $n=2N+1+2M+1$, i.e., $m=2M+1$, for $\alpha=-2N-1$, 
the proof is similar and will be omitted. 
Hence \eqref{id:bmonejfactbe1} follows. 
The identity \eqref{id:bmonejfactal1} holds 
by applying the same procedure, which completes the proof.
\end{proof}

A different way to guess and then obtain the factorization for 
the big $-1$ Jacobi polynomials for non-standard parameters throughout is to 
use the following two step procedure. The first 
step is to use Lemma \ref{lem:1} in order to prove the identities 
\eqref{id:bmonejfactal0}, \eqref{id:bmonejfactbe0}. Once 
we have that, it is straightforward to find, e.g., for $m\ge 0$, one has 
\[
Q_{2N+1+m}^{(0)}(x;-2N-1,\beta,c)=
Q_{2N+1}^{(0)}(x;-2N-1,\beta,c) P_m(x),
\]
where $P_m\in \mathbb P$ of degree $m$.
In the second step, we obtain a new recurrence 
relation for these new polynomials and comparing 
the recurrence coefficients, we prove this new 
polynomial sequence is related to the original one.
\begin{lemma}\label{lem:2}
For any $n, N\in \N$, $\alpha, \beta, 
c\in \mathbb C$, $c\ne\pm 1$, with $\alpha=-2N-1$ 
(resp.~$\beta=-2N-1$), the following recurrence relation holds
for $n\ge0$:
\[
x Q^{(0)}_{n}(x;2N+1,\beta,-c)=
-Q^{(0)}_{n+1}(x;2N+1,\beta,-c)+
\widehat b_n Q^{(0)}_{n}(x;2N+1,\beta,-c)+
\widehat u_n Q^{(0)}_{n-1}(x;2N+1,\beta,-c),
\]
where 
\[
\widehat b_n=-b_{n+2N+1}(-2N-1,\beta,c)=b_{n}(2N+1,\beta,-c),\qquad 
\widehat u_n=u_{n+2N+1}(-2N-1,\beta,c)=
u_{n}(2N+1,\beta,-c).
\]
Respectively, the following recurrence relation holds for $n\ge0$:
\[
x Q^{(0)}_{n}(x;\alpha, 2N+1,-c)=
Q^{(0)}_{n+1}(x;\alpha,2N+1,-c)+
\widetilde b_n Q^{(0)}_{n}(x;\alpha, 2N+1,-c)+
\widetilde u_n Q^{(0)}_{n-1}(x;\alpha, 2N+1,-c),
\]
where 
\[
\widetilde b_n=b_{n+2N+1}(\alpha,-2N-1,c)=
b_{n}(\alpha,2N+1,-c),\qquad 
\widetilde u_n=u_{n+2N+1}(\alpha,-2N-1,c)=
u_{n}(\alpha,2N+1,-c).
\] 
\end{lemma}
\begin{proof}
We start by considering $\alpha=-2N-1$, then one has  
$u_{2N+1}=0$. In fact, 
\[
x Q^{(0)}_{2N+1}(x;2N+1,\beta,c)=(x^2-1)^N(x-1).
\]
If we consider $n\ge m+2N+1$, then
\[
Q^{(0)}_{2N+1+m}(x;-2N-1,\beta,c)=
Q^{(0)}_{2N+1}(x;-2N-1,\beta,c) (-1)^m
Q^{(0)}_{m}(-x;2N+1,\beta,-c),
\]
and therefore the recurrence relation for 
these polynomials can be written as
\[\begin{split}
x (-1)^m Q^{(0)}_{m}(-x;2N+1,\beta,-c)=&
(-1)^{m+1} Q^{(0)}_{m+1}(-x;2N+1,\beta,-c)+
b_{m+2N+1}(-1)^{m} Q^{(0)}_{m}(-x;2N+1,\beta,-c)\\
&+u_{m+2N+1}(-1)^{m-1} Q^{(0)}_{m-1}(-x;2N+1,\beta,-c).
\end{split}\]
After some straightforward calculations 
the result follows. For $\beta=-2N-1$, the derivation is analogous which completes the proof. 
\end{proof}


Note that if $u_{2N+1}=0$ for some $N\in \mathbb N$ 
then the orthogonality property for the big $-1$ 
Jacobi polynomials holds for polynomials of 
degree less or equal to $2N+1$. 
In the next result, we construct, thanks to 
the factorization of the big $-1$ polynomials, 
a new property of orthogonality that is valid 
for the big $-1$ Jacobi polynomials even when $n>2N+1$, i.e.,~we can extend the property of 
orthogonality for the big $-1$ polynomials for 
non-standard degenerate parameters.

\begin{theorem}
[Orthogonality of big $-1$ Jacobi polynomials for non-standard parameters]
\label{thm:1}
Let $N\in \N_0$, $c\in\mathbb C$, 
$c\ne\pm 1$. 
Define the linear operator $\bf u$ 
\cite[Section 4]{MR2931336} as
\[
\langle {\bf u}, p q\rangle:=\int_{[-c,-1]\cup[1,c]} p(x) q(x)\,
\dfrac{x}{|x|}(x+1)(c-x)(x^2-1)^{\frac12(\alpha-1)}(c^2-x^2)^{\frac12(\beta-1)}\, \mathrm dx.
\]
The norm with respect to the linear functional 
${\bf u}\in \mathbb P'$ is given by $h_n\!:=\! h_n(\alpha,\beta,c)$,
\begin{equation}
\label{eq:bonejn}
\hspace{0.3cm}h_n(\alpha,\beta,c)\!:=\!\langle {\bf u}, Q_n^{(0)}Q_n^{(0)}\rangle\!=
\!\langle {\bf u}, 1\rangle\!
\left\{ \begin{array}{ll}
\displaystyle \!\dfrac{2(c^2-1)^n(\frac12 n)!(\frac12(\alpha+1),\frac12(\beta+1))_{\frac12n}}
{(\frac12(\alpha+\beta)+1)_n(\frac12(\alpha+\beta+n)+1)_{\frac12n}},\ \mathrm{if}\ n\ \mathrm{even},\\[18pt]
\displaystyle \!\dfrac{2(c\!-\!1)^{n-1}(c\!+\!1)^{n+1}(\frac12(n\!-\!1))!(\frac12(\alpha\!+\!1)
,\frac12(\beta
\!+\!1))_{\frac12(n+1)}}
{(\frac12(\alpha\!+\!\beta)+1)_n(\frac12(\alpha\!+\!\beta\!+\!n\!+\!1))_{\frac12(n+1)}}, \ \mathrm{if}\ n\ \mathrm{odd}.
\end{array} \right.
\end{equation}
Then, 
the polynomial sequences 
\begin{eqnarray}
&&(Q_n^{(0)}(x;-2N-1,\beta,c)), \quad
(Q_n^{(0)}(x;\alpha,-2N-1,c)),
\end{eqnarray}
are orthogonal with respect to the bilinear forms
\begin{eqnarray}
&&\label{case1}\langle p, q\rangle_1=
{\mathscr L}_0(p,q)+
\lambda_N(\beta)\langle {\bf u}_1, (\tau_{\alpha}^{2N+1} p)
(\tau_{\alpha}^{2N+1} q)\rangle,\\
&&\label{case2}\langle p, q\rangle_2=
{\mathscr L}_1(p,q)+
\lambda_N(\alpha)\langle {\bf u}_2, (\tau_{\beta}^{2N+1} p)
(\tau_{\beta}^{2N+1} q)\rangle,
\end{eqnarray}
respectively.  Here 
$\tau_\alpha$, $\tau_\beta$ are linear 
operators defined such that 
\[
\tau_\alpha[Q^{(0)}_n(x;\alpha,\beta,c)]=
Q_{n-1}(-x;\alpha+2,\beta,-c),
\]
\[\tau_\beta[Q^{(0)}_n(x;\alpha,\beta,c)]=
Q_{n-1}(x;\alpha,\beta+2,-c),
\]
and
\[
\langle {\bf u}_1,f(x,\alpha,\beta,c) \rangle:=\langle {\bf u}, f(-x,-\alpha,\beta,-c)\rangle,
\quad 
\langle {\bf u}_2,f(x,\alpha,\beta,c) \rangle:=\langle {\bf u}, f(x,\alpha,-\beta,-c)\rangle,
\]
$\mathscr L_0$ and $\mathscr L_1$ are two 
moment linear functionals, and 
\[
\lambda_{N}(\alpha)=\frac12
h_{2N+1}(-2N-1,\alpha,c),\quad
\lambda_{N}(\beta)=\frac12
h_{2N+1}(-2N-1,\beta,c).
\]

\end{theorem}
\begin{proof}
Let us consider the situation corresponding to \eqref{case1}, i.e., $\alpha=-2N-1$, with $N\in \mathbb N$, $\beta\in \mathbb C$.
We need to prove the the properties of orthogonality $\langle Q_n^{(0)}, Q_m^{(0)}
\rangle_1= 0$ for all $m<n$, $\langle Q_n^{(0)}, Q_n^{(0)}\rangle_1\ne  0$ for all 
$n$ in order to prove that $(Q_n^{(0)})$ is a monic orthogonal polynomial sequence 
with respect to $\langle,\rangle_1$. Let us prove this in three steps: (i) $m\le n<2N+1$; (ii) $m<2N+1\le n$;  (iii) $2N+1\le n\le m$.
(i):~If $m\le n<2N+1$ 
the coefficient $u_k$ of the recurrence relation is non-zero for $k=0, 1, \dots, 2N$,
so by the Favard result (see \cite{chi1}), there exists  a moment linear functional, 
namely $\mathscr L_0$ 
such that the big $-1$ Jacobi polynomials $(Q_j^{(0)})_{j=0}^{2N}$
are orthogonal with respect to it.
Due to the 
fact that $\tau_\alpha^{2N+1}(Q_j)=0$ for $j=0, 1, ...., 2N$, 
therefore the property of orthogonality holds in this case, i.e.,
\[
\langle Q_n^{(0)}, Q_m^{(0)}\rangle_1=
\langle {\bf u},   Q_n^{(0)} Q_m^{(0)}\rangle= h_n(-2N-1,\beta,c)\delta_{n,m}, 
\]
where $\delta_{n,m}$ is the Kronecker (delta) symbol. 
(ii):~If $m<2N+1\le n$, then 
$\mathscr L_0(Q^{(0)}_n,Q^{(0)}_m)=0$; at the same 
time $\tau_\alpha^{2N+1}(Q_j)=0$ for $j=0, 1, ...., 2N$, 
therefore the property of orthogonality holds in this case.
(iii):~If $2N+1\le n\le m$, then 
$\mathscr L_0(Q^{(0)}_n,Q^{(0)}_m)=0$; and since 
\[
\tau_\alpha^{2N+1} Q^{(0)}_n(x;-2N-1,\beta,c)=
Q^{(0)}_{n-(2N+1)}(-x;2N+1,\beta,-c),
\]
is orthogonal with respect to the linear functional 
$\bf u$, one has 
\[\begin{split}
\langle Q_n^{(0)}, Q_m^{(0)}\rangle_1=&\lambda_N(\beta)
\langle {\bf u},  \tau_\alpha^{2N+1} Q^{(0)}_n(x;-2N-1,\beta,c)
\tau_\alpha^{2N+1} Q^{(0)}_m(x;-2N-1,\beta,c) 
\rangle\\[-0.5mm]
=& \lambda_N(\beta) \langle {\bf u},  Q^{(0)}_{n-2N-1}(-x;2N+1,\beta,-c)
Q^{(0)}_{m-2N-1}(-x;2N+1,\beta,-c)=
h_n(-2N-1,\beta,c)\delta_{n,m}, 
\end{split}\]
and the property of orthogonality 
holds in this situation.
Moreover, since the polynomials are monic, 
it is straightforward to check 
using the recurrence relation \eqref{ttrrbmonej} that
\begin{equation} \label{eq:3.18} 
h_n=\langle {\bf u}, xQ^{(0)}_{n}(x)Q^{(0)}_{n-1}(x)\rangle=
u_n \langle {\bf u}, xQ^{(0)}_{n-1}(x)Q^{(0)}_{n-1}(x)\rangle
=u_n h_{n-1}.
\end{equation}
Furthermore, the norm  squared \eqref{eq:bonejn} and the recurrence coefficient $u_n$ 
given in  \eqref{ttrrbmonej} satisfies \eqref{eq:3.18}. 
The situation which corresponds to \eqref{case2}, i.e., $\beta=-2N-1$, with $N\in \mathbb N$, $\alpha\in \mathbb C$ is similar so we will omit its proof. Hence the result holds. 
\end{proof}

\begin{remark}\label{rem:2}
Note that if $\alpha=-2N-1$ and $\beta\in \mathbb C$ and 
due to \eqref{id:bmonejfactal0} one must consider the 
moment linear form \cite{MR2558142}
\[
{\mathscr L_0}(p,q)=
\sum_{j=0}^{N} \lambda_j \,p^{(j)}(1)\,q^{(j)}(1)
+\sum_{j=0}^{N-1} \mu_j \,p^{(j)}(-1)\,q^{(j)}(-1),
\]
where $\lambda_j, \mu_j$ are positive. So, taking into 
account \eqref{id:bmonejfactbe1}, we have 
\[
{\mathscr L}_0(Q^{(0)}_j(x),\pi(x))=0,\quad j\ge 2N+1,
\]
for any $\pi\in\mathbb P$.
Moreover, 
by construction, if $0\le j<2N+1$ one has
\[
\langle Q^{(0)}_j(x), Q^{(0)}_j(x)\rangle_1=
{\mathscr L}_0\left(Q^{(0)}_j(x)p,Q^{(0)}_j(x)\right)=
\langle {\bf u}, Q_j^{(0)}Q_j^{(0)}\rangle=h_j\ne 0.
\]
If $j\ge 2N+1$ and taking into account 
Lemma \ref{lem:2} one has
\[\begin{split}
\langle Q^{(0)}_j(x), Q^{(0)}_j(x)\rangle_1=&
\lambda_N(\beta) \langle {\bf u}, 
Q_{j-2N-1}^{(0)}(-x;2N+1,\beta,-c)
Q_{j-2N-1}^{(0)}(-x;2N+1,\beta,-c)\rangle\\
=&\lambda_N(\beta) h_{j-2N-1}(2N+1,\beta,-c)=h_{j}(-2N-1,\beta,c).
\end{split}
\]
Observe the last identity must be considered as the 
formal limit:
\[
\tfrac12\lim_{\epsilon\to 0}h_{2N+1}(-2N-1+\epsilon,\beta,c) 
h_{j-2N-1}(2N+1+\epsilon,\beta,-c)=
\lim_{\epsilon\to 0}
h_{j}(-2N-1+\epsilon,\beta,c).
\]
For $\beta=-2N-1$, it is analogous since 
in \eqref{id:bmonejfactbe0} one must consider the 
moment linear form \cite{MR2558142}
\[
{\mathscr L_1}(p,q)=
\sum_{j=0}^{N} \kappa_j p^{(j)}(1)q^{(j)}(-c)
+\sum_{j=0}^{N-1} \ell_j p^{(j)}(-1)q^{(j)}(c),
\]
where $\kappa_j, \ell_j$ are positive. Moreover, in a similar way 
\[
\dfrac12\lim_{\epsilon\to 0}h_{2N+1}(\alpha,-2N-1+\epsilon,c) 
h_{j-2N-1}(\alpha,2N+1+\epsilon,-c)=
\lim_{\epsilon\to 0}
h_{j}(\alpha,-2N-1+\epsilon,c).
\]
\end{remark}
\section{Conclusion and future work}
We have studied the big $-1$ Jacobi polynomials, 
which are a $q\to-1$ limit of the Bannai--Ito polynomials, 
when at least one 
of the parameters $\alpha$, $\beta$, are non-standard
and the coefficient $u_n$ of the 
three-term recurrence relation \eqref{ttrrbmonej} 
is equal to zero 
for a certain index, namely $N$, 
i.e., $\alpha=-2N-1$ or $\beta=-2N-1$ where $N$ is 
a positive integer.
We have also obtained the Gauss hypergeometric 
representation, its factorization, and the property of 
orthogonality 
using similar techniques that were used in \cite{cola1}. 

In \cite[Section 4]{cola2}, a similar procedure 
for the big $q$-Jacobi polynomials
was applied for non-standard parameters for 
which the $u_n$ coefficient of its recurrence 
relation is equal to 
zero 
for some $n\in\mathbb N_0$.
Along these lines we are investigating  
the method 
for obtaining the orthogonality property of the 
big $-1$ Jacobi polynomials for non-standard 
parameters following a procedure analogous to 
the one followed by 
T.~E.~P\'erez and M.~A.~Piñar in \cite{mr1405981} 
for the generalized Laguerre polynomials. 
In \cite{cola2} the authors also studied 
the $q$-Racah polynomials for non-standard 
parameters for which the $u_n$ coefficient vanishes 
for some $n\in\mathbb N_0$ and since the Bannai--Ito 
polynomials can be obtained from the $q$-Racah polynomials (\cite[\S14.2]{Koekoeketal}) by taking 
an analogous limit $q\to -1$ (see \cite{MR2931336}), 
it makes sense to perform a similar
study for the Bannai--Ito polynomials. 
We are also 
working on obtaining explicit 
hypergeometric representations,
the factorization, and the property of orthogonality
for the Bannai--Ito polynomials and other families of the 
continuous $-1$ hypergeometric orthogonal polynomial 
scheme for non-standard parameters. 
\bibliographystyle{amsplain}
\bibliography{refbib}
\end{document}